\documentclass{amsart}
\usepackage{amssymb}
\usepackage{hyperref}
\usepackage{multicol}
\usepackage{multirow}
\usepackage{colordvi}
\usepackage{graphicx}
\usepackage{pgf,epsfig, epsf}
\numberwithin{equation}{section}
\theoremstyle{plain}
\newtheorem{Theorem}{Theorem}[section]

\newtheorem{Algorithm}[Theorem]{Algorithm}

\newtheorem{Proposition}[Theorem]{Proposition}
\newtheorem{Lemma}[Theorem]{Lemma}
\theoremstyle{definition}

\newtheorem{Example}[Theorem]{Example}
\newtheorem{Definition}[Theorem]{Definition}

\copyrightinfo{}{}
\newcounter{FNC}[page]
\def\fauxfootnote#1{{\addtocounter{FNC}{2}\Magenta{$^\fnsymbol{FNC}$}%
     \let\thefootnote\relax\footnotetext{\Magenta{$^\fnsymbol{FNC}$#1}}}}

\newcommand{\Z}{\mathbb Z}
\newcommand{\R}{\mathbb R}

\newcommand{\conv}{\operatorname{conv}}

\newcommand{\w}{\operatorname{w}}
\newcommand{\nls}{\operatorname{nls_\Delta}}
\newcommand{\ls}{\operatorname{ls_\Delta}}
\newcommand{\lstwo}{\operatorname{ls_{\Delta_2}}}
\newcommand{\lss}{\operatorname{ls_\square}}

\newcommand{\lan}{\langle}
\newcommand{\ran}{\rangle}

\title{Lattice Size of Width One Lattice Polytopes in $\mathbb{R}^3$}

\author{Abdulrahman Alajmi}
\address{Department of Mathematical Sciences\\
        Kent State University\\
        800 E. Summit st., Kent, OH 44242, USA}
\email{aalajmi@kent.edu}

\author{Jenya Soprunova}
\address{Department of Mathematical Sciences\\
        Kent State University\\
        800 E. Summit st., Kent, OH 44242, USA}
\email{esopruno@kent.edu}
\urladdr{http://www.math.kent.edu/~soprunova/}

\subjclass[2020]{11H06, 52B20, 52C07}
\keywords{Lattice size, lattice width, lattice polygons, generalized basis reduction.}

\begin{document}

\begin{abstract} The lattice size $\operatorname{ls_\Delta}(P)$ of a lattice polytope $P$ is a geometric invariant, which was formally introduced  in relation to the problem of bounding the total degree and the bi-degree of the defining equation of an algebraic curve, but appeared implicitly earlier in geometric combinatorics. In this paper, we show that for an empty lattice polytope $P\subset\mathbb{R}^3$ there exists a reduced basis  of $\mathbb{Z}^3$ which computes  its lattice size $\operatorname{ls_\Delta}(P)$. This leads to a fast algorithm for computing $\operatorname{ls_\Delta}(P)$ for such $P$. We also extend this result to another class of lattice width one polytopes $P\subset\mathbb{R}^3$. We then provide a counterexample demonstrating that  this result does not hold true for an arbitrary lattice polytope $P\subset\mathbb{R}^3$ of lattice width one.
\end{abstract}

\maketitle

%

\section{Introduction}
In this paper, we apply basis reduction to the problem of computing the lattice size of lattice polytopes $P\subset\R^3$ whose lattice width is equal to one. In particular, we provide a basis-reduction algorithm for computing the lattice size of empty lattice polytopes $P\subset\R^3$. 

A reduced basis is a basis of the integer lattice that consists of vectors that are short with respect to a given norm. The problem of finding such a basis is a complex and important problem that has been extensively studied, particularly owing to its wide-ranging applications in cryptology (as illustrated in~\cite{Barvinok} and Chapter IV of~\cite{Galbraith}).  The LLL algorithm~\cite{LLL}, introduced by Lenstra, Lenstra, and Lovász in 1982, marked a significant breakthrough in the lattice basis reduction theory in the case of the inner-product norm. Subsequently, various other algorithms have been developed, primarily focusing on addressing the case of the inner-product norm.

In~\cite{LovScarf}, Lovász and Scarf generalized the LLL reduction algorithm to the case of the general norm.  In the case of dimension 2, this algorithm was
analyzed by Kaib and Schnorr in~\cite{KaibSchnorr}. In~\cite{HarSopr} this analysis was extended to dimension 3, providing a fast algorithm for basis reduction with respect to the general norm. This algorithm 
was then applied in~\cite{HarSopr} to computing the lattice size, as we elaborate further.

The lattice size was formally introduced in~\cite{CasCools} in the context  of simplification of an equation defining an algebraic curve. It appeared implicitly earlier
 in~\cite{Arnold, BarPach, BrownKasp, LagZieg, Schicho}, and was further studied in~\cite{Abdul-thesis,ACKS, HarSopr, HarSoprTier, Sopr}. We next reproduce its definition.

Denote by $\Delta=\conv\{0,e_1,\dots, e_d\}\subset\R^d$ the standard simplex, where $0\in\R^d$ is the origin and $(e_1,\dots, e_d)$ is the standard basis of $\R^d$.
See the next section for the definition of  the group of affine unimodular maps $T\colon\R^d\to \R^d$, denoted by  ${\rm AGL}(d,\Z)$.
\vspace{.1cm}
\begin{Definition}~\label{D:lattice-size} The {\it lattice size} $\ls(P)$ of a lattice polytope  $P\subset\R^d$ with respect to the standard simplex $\Delta$ is the smallest $l\geq 0$ such that $T(P)$ is contained in the $l$-dilate  $l\Delta$  for some affine unimodular map $T$.
Equivalently, if we set
$$l_1(P)=\max\limits_{(x_1,\dots,x_d)\in P} (x_1+\cdots+x_d)-\min\limits_{(x_1,\dots,x_d)\in P} x_1-\cdots-\min\limits_{(x_1,\dots,x_d)\in P} x_d,$$
then $\ls(P)=\min\{l_1(T(P))\mid T\in{\rm AGL}(d,\Z)\}$.
\end{Definition}

Replacing in Definition~\ref{D:lattice-size} the standard simplex $\Delta$ with the unit cube $\square=[0,1]^d$, we  obtain the definition of  the lattice size with respect to the unit cube, denoted by $\lss(P)$. 
See Example~\ref{E:lattice-size} for an illustration of both definitions.

In~\cite{CasCools} the authors provided an algorithm for computing both  lattice sizes, $\ls(P)$ and $\lss(P)$, of a plane lattice polygon $P\subset\R^2$. This algorithm is based on a procedure for mapping a polygon inside a small multiple of the standard simplex, introduced by Schicho in~\cite{Schicho}. In this algorithm, called the ``onion skins'' algorithm, one passes recursively from a lattice polygon $P$ to the convex hull of the interior lattice points of $P$. This algorithm applies only to lattice polygons in the plane, and its extension to higher dimensions is not feasible. 

It was subsequently shown in~\cite{HarSopr, HarSoprTier} that basis reduction provides a faster way of computing both lattice sizes, $\ls(P)$ and $\lss(P)$ for $P\subset\R^2$. This approach applies not only to lattice polygons, but also to arbitrary plane convex bodies.  Furthermore, it was shown in~\cite{HarSopr} that basis reduction computes $\lss(P)$ for  $P\subset\R^3$. However, as demonstrated by a counterexample in~\cite{HarSoprTier}, basis reduction does not necessarily compute $\ls(P)$ for  $P\subset\R^3$.

In this paper, we focus on the case of width one lattice polytopes $P\subset\R^3$ and, in particular,  of empty lattice polytopes $P\subset\R^3$.
Recall that a lattice polytope $P$ is {\it empty} if its  only lattice points are its vertices. It was shown in~\cite{Scarf} that any empty lattice polytope $P\subset\R^3$ has lattice width one. 
In our main result, Theorem~\ref{T:empty-polytopes},  we prove that for an empty lattice polytope $P\subset\R^3$ there exists a reduced basis 
that computes its lattice size $\ls(P)$. Algorithm~\ref{A:compute-ls}, based on this theorem, uses basis reduction to compute the lattice size $\ls(P)$ of an empty lattice polytope  $P\subset\R^3$.

%

In Section~\ref{S:another-class},  we consider another class of lattice polytopes $P\subset\R^3$ of lattice width one and show in Theorem~\ref{T:another-class}  that there exists a reduced basis that computes $\ls(P)$. 

Based on our results, it is natural to ask whether for any lattice polytope $P\subset\R^3$ of lattice width one there exists a reduced basis that computes $\ls(P)$.
In Section~\ref{S:experimentation}, we provide an example that demonstrates  that the answer to this question is negative.

\section{Definitions and First Lemmas}
A point in $\R^d$ is a {\it lattice point} if it belongs to $\Z^d\subset\R^d$. A {\it lattice polytope} $P\subset\R^d$ is a convex polytope all of whose vertices are lattice points. A segment connecting two lattice points is called a {\it lattice segment}.  A lattice segment is {\it primitive} if its only lattice points are its endpoints. A vector is {\it primitive} if it is a direction vector of a primitive segment.  

Recall that a {\it unimodular matrix} $A\in{\rm GL}(d,\Z)$ is a square matrix of size $d$ with integer entries that satisfies $\det A=\pm 1$. Further, ${\rm AGL}(d,\Z)$ is the group of affine unimodular transformations $T\colon\R^d\to\R^d$ of the form 
$T(x)=Ax+h$, where $A\in{\rm GL}(d,\Z)$ and $h\in\Z^d$. Note that the rows of an integer matrix $A$ of size $d$ form a basis of $\Z^d$ if and only if $A$ is unimodular.
Another standard fact is that a primitive vector $h\in\Z^d$ can be extended to a basis of $\Z^d$.

We say that two lattice polytopes in $P, Q\subset\R^d$ are {\it lattice equivalent} if $Q=T(P)$ for some $T\in{\rm AGL}(d,\Z)$.
We will write $AP$ for the image of $P$ under the map $T\colon \R^d\to\R^d$ defined by $T(x)=Ax$.

For a lattice polytope $P\subset\R^d$ 
 the {\it lattice width of  $P$ in the direction of} $h\in\R^d$ is
$$\w_h(P)=\max\limits_{x\in P} \lan h, x\ran- \min_{x\in P} \lan h, x\ran,
$$
where $ \lan h, x\ran$ denotes the standard inner product in $R^d$. Then the {\it lattice width} $\w(P)$  of $P$ is the minimum of $\w_h(P)$ over all non-zero primitive vectors $h\in \Z^d$.

We next illustrate the definition of the lattice size provided in the introduction.

  \begin{Example}\label{E:lattice-size} 
Let  $P=\conv\{(0,0), (1,0), (3,1), (6,3), (4,3)\}$, as depicted in Figure~\ref{F:ls_example}. 
Define
 $T(x,y)=\begin{bmatrix}1&-1\\-1&2\end{bmatrix}\cdot \begin{bmatrix}x\\ y\end{bmatrix}+\begin{bmatrix}0\\ 1\end{bmatrix}$. Then 
 $$T(P)=\conv\{(0,1), (1,0), (2,0), (3,1), (1,3)\}.$$  
 We observe that  $T(P)\subset 3\square$ and $T(P) \subset 4\Delta$. Furthermore, since $P$ contains three interior lattice points while $2\square$ and $3\Delta$ contain only one, it follows that it is impossible to unimodularly map $P$ within $2\square$ or $3\Delta$. Therefore, we conclude that $\lss(P)=3$ while $\ls(P) = 4$. 
 \end{Example}
\begin{figure}[h]
\begin{center}
\includegraphics[scale=.45]{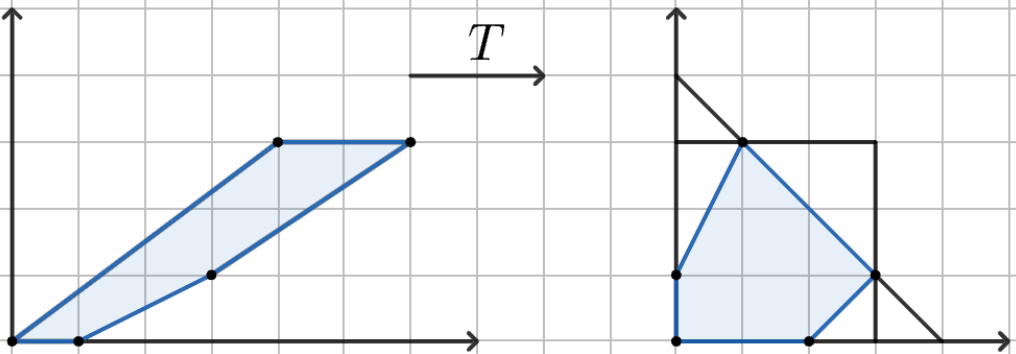} 
\caption{Example~\ref{E:lattice-size}.}
\label{F:ls_example} 
\end{center}
\end{figure}

\begin{Definition} For $P\subset\R^d$, {\it the naive lattice size}  $\nls(P)$ is the smallest $l$ such that $T(P)\subset l\Delta$, where $T\colon\R^d\to\R^d$ is a composition of matrix multiplication by a 
diagonal matrix $A$ with entries $\pm 1$ on the main diagonal, and a lattice translation. Further, for $d=2$, let
\begin{eqnarray*}
l_1(P)&:=&\max\limits_{(x,y)\in P}(x+y)-\min\limits_{(x,y)\in P}x-\min\limits_{(x,y)\in P}y,\\
l_2(P)&:=&\max\limits_{(x,y)\in P} x+\max\limits_{(x,y)\in P}y-\min\limits_{(x,y)\in P}(x+y),\\
l_3(P)&:=&\max\limits_{(x,y)\in P} y-\min\limits_{(x,y)\in P}x +\max\limits_{(x,y)\in P} (x-y),\\
l_4(P)&:=&\max\limits_{(x,y)\in P} x-\min\limits_{(x,y)\in P}y+\max\limits_{(x,y)\in P}(y-x).
\end{eqnarray*}
Then in the case $d=2$ the naive lattice size $\nls(P)$ is the smallest of these four values. For $P\subset\R^3$ one can similarly write $\nls(P)$ as the minimum of eight $l_i(P)$, each corresponding to a vertex of the unit cube.
\end{Definition}

\begin{Example}\label{E:def-nls}
Consider $P=\conv\{(0,0), (3,0), (3,2), (2,3), (0,1)\}$, as depicted in Figure~\ref{F:def-nls}.  Then $l_1(P)=5$, $l_2(P)=6$, $l_3(P)=6$, and $l_4(P)=4$, which implies that $\nls(P)=l_4(P)=4$.
 \begin{figure}[h]
\begin{center}
\includegraphics[scale=.33]{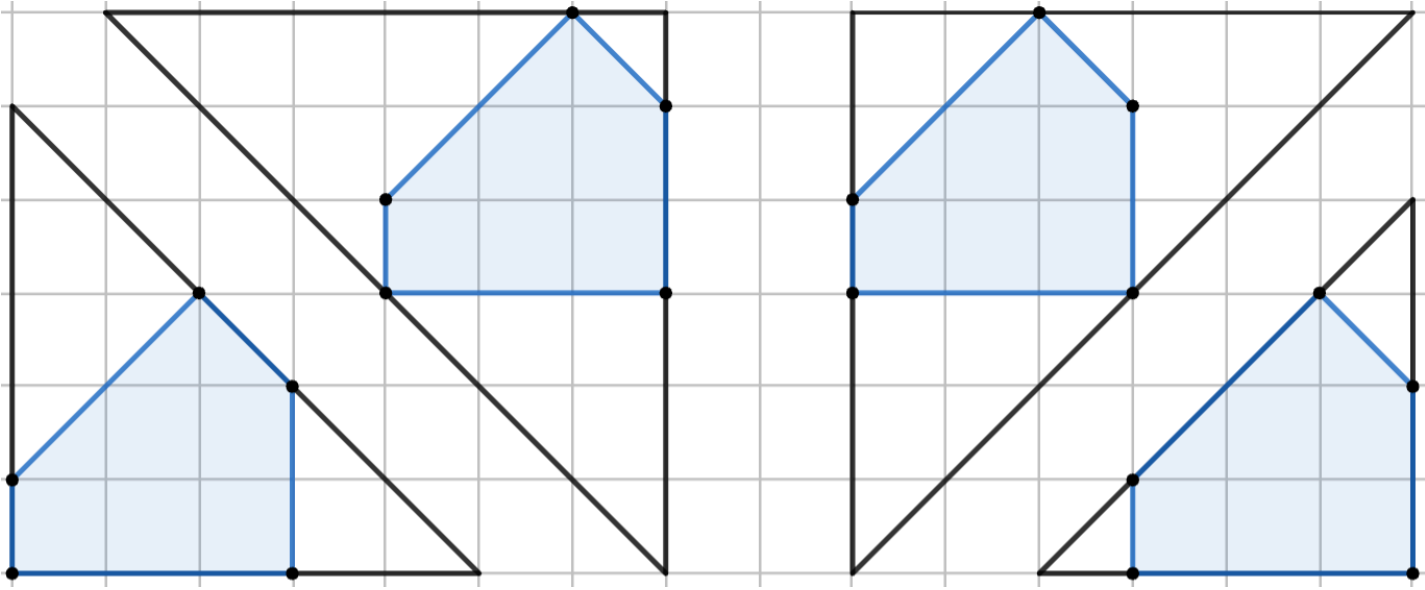} 
\caption{Example~\ref{E:def-nls}}
\label{F:def-nls} 
\end{center}
\end{figure}
\end{Example}

\begin{Definition}
We say that a basis $(h_1,\dots, h_d)$ of $\Z^d$ computes $\ls(P)$ if for matrix $A$ with rows $h_1,\dots, h_d$ we have $\ls(P)=\nls(AP)$.
\end{Definition}

Let $(-P)$ be the reflection of $P$ in the origin, and define  $K$ to be the polar dual of the Minkowski sum of $P$ with $(-P)$, that is, $K:=(P+(-P))^{\circ}$.
Then $K$ is origin-symmetric and convex and it defines a norm on $\R^d$ by
$$\Vert h\Vert _K=\inf\{\lambda>0\mid h/\lambda \in K\}.$$
For details see, for example, \cite{Barvinok}. We then have
$$\Vert h\Vert_K =\inf\{\lambda>0\mid \lan h, x\ran\leq\lambda {\rm\ for\ all\ }  x\in K^\circ\}=\max\limits_{x\in K^{\circ}} \lan h, x\ran =\frac{1}{2}\w_h(K^{\circ})=\w_h(P).
$$
In what follows we will often write  $\Vert h\Vert$ for the lattice width $\w_h(P)$.

\begin{Definition}\label{D:reduced} Let $P\subset\R^2$. We say that a basis $(h_1, h_2)$ of the integer lattice $\Z^2\subset\R^2$ is {\it reduced (with respect to $P$)} if 
for the norm defined above we have
\begin{enumerate}
\item $\Vert h_1 \Vert\leq \Vert h_2 \Vert$ and
\item $\Vert h_1\pm h_2 \Vert \geq  \Vert h_2\Vert$.
\end{enumerate}
\end{Definition}

Note that for the sake of simplicity in this definition we refer to such a basis as reduced with respect to $P$, while formally one needs to say that it is reduced with respect to the centrally-symmetric body $K=(P+(-P))^\circ$, as was done in~\cite{HarSopr}. We will maintain this slight abuse of notation throughout the paper.

It was shown in~\cite{HarSopr} that if the standard basis $(e_1,e_2)$ is reduced with respect to a convex body $P\subset\R^2$, then $\lss(P)=\w_{e_2}(P)$. This result was extended to the lattice size with respect 
to the standard simplex $\Delta\subset\R^2$ in~\cite{HarSoprTier}, where it was shown that  if the standard basis  $(e_1,e_2)$ is reduced with respect to $P$, then it computes $\ls(P)$, that is, $\ls(P)=\nls(P)$. Since one can find a reduced basis
using the generalized Gauss reduction algorithm~\cite{KaibSchnorr,LovScarf}, these results provide a fast way of finding $\ls(P)$ and $\lss(P)$ for plane convex bodies $P$. As explained in~\cite{HarSopr, HarSoprTier}, this method outperforms the ``onion skins" algorithm of~\cite{CasCools}.

\begin{Example}
For the polygon $P$ in Example~\ref{E:def-nls} we have $\w_{e_1}(P)=3$, $\w_{e_2}(P)=3$, $\w_{e_1+e_2}(P)=5$, and $\w_{e_1-e_2}(P)=4$, which implies that the standard basis is reduced with respect to $P$.
We conclude that $\lss(P)=\w_{e_2}(P)=3$ and $\ls(P)=\nls(P)=l_4(P)=4$.
\end{Example}

Since the condition $\Vert h_1\pm h_2 \Vert \geq  \Vert h_2\Vert$ is equivalent to requiring that  $\Vert mh_1+ h_2 \Vert \geq  \Vert h_2\Vert$ for $m\in\Z$ (see~\cite{HarSopr}), the following definition is a natural extension of 
the previous one.

\begin{Definition}\label{D:Reduced} Let $P\subset\R^3$. We say that a basis $(h_1, h_2, h_3)$ of $\Z^3$ is reduced (with respect to $P$) if the following three conditions are satisfied.
\begin{enumerate}
\item $\Vert h_1 \Vert \leq  \Vert h_2 \Vert \leq \Vert h_3 \Vert $;
\item $\Vert h_1\pm h_2\Vert \geq  \Vert h_2 \Vert$;
\item $\Vert mh_1+nh_2+h_3\Vert \geq \Vert h_3\Vert$ for all $m,n\in\Z$.
\end{enumerate}
\end{Definition}

An algorithm for finding a basis which is reduced with respect to  a convex body  $P\subset\R^3$ is provided in~\cite{HarSopr}.
It is also shown in~\cite{HarSopr} that if the standard basis $(e_1,e_2,e_3)$ is reduced with respect to $P\subset\R^3$, then $\lss(P)=\w_{e_3}(P)$. 
This result does not extend to the lattice size with  respect to the standard simplex. An example provided in~\cite{HarSoprTier} demonstrates that there exists a 
lattice polytope $P\subset\R^3$ such that there is no reduced basis that computes $\ls(P)$.

Let $P\subset\R^3$ be a lattice polytope. We record a few observations which are 3D versions of Lemmas 2.5 and 2.6 in~\cite{HarSoprTier} and Lemmas 2.1 and 2.2 in~\cite{ACKS}.

\begin{Lemma}\label{L:applyA}
Let $A\in{\rm GL}(3,\Z)$ have rows $h_1,h_2, h_3$. Then
\begin{enumerate}
\item For  $h\in\R^3$ we have  $\w_{h}(AP)=\w_{A^{T}h}(P)$.
\item For $i=1,2,3$ we have $\w_{e_i}(AP)=\w_{h_i}(P)$. 
\item We have $\w_{e_i}(P)\leq l_1(P)$.
\item Let $h$ be $h_1,h_2,h_3, h_1+h_2,h_1+h_3,h_2+h_3$, or $h_1+h_2+h_3$. Then $\w_h(P)\leq l_1(AP)$.
\item Let $h$ be as in {\rm (4)} and suppose that $l_1(AP)<l_1(P)$.  Then $\w_{h}(P)<l_1(P)$.
\end{enumerate}
\end{Lemma}

\begin{proof}
For (1) we have
$$\w_{h}(AP)=\max\limits_{x\in P} \lan h, Ax\ran - \min_{x\in P} \lan h, Ax\ran=\max\limits_{x\in P} \lan A^{T}h,x\ran- \min_{x\in P} \lan A^{T}h,x\ran=\w_{A^{T}h}(P),
$$
and (2) follows as then $\w_{e_i}(AP)=\w_{A^{T}e_i}(P)=\w_{h_i}(P)$.
To prove (3), denote $l_1=l_1(P)$. Then $P\subset l_1\Delta$ and hence 
$$\w_{e_i}(P)\leq \w_{e_i}(l_1\Delta)=l_1.$$
Next, we have $\w_{h_i}(P)=\w_{e_i}(AP)\leq l_1(AP)$.
The same argument works if we replace $h_i$ with the sum of any nonempty collection of rows of $A$, while replacing $e_i$ with the corresponding sum of the basis vectors, and (4) follows.
Finally, (5) is a trivial corollary of (4). 
\end{proof}

\begin{Lemma}\label{L:sum-of-rows} Let $h_1,h_2$, and $h_3$ be the rows of $A\in{\rm GL}(3,\Z)$. Then
\begin{itemize}
\item[(a)] $l_1(AP)=\max_{x\in P} \lan h_1+h_2+h_3, x\ran-\min_{x\in P} \lan h_1, x\ran-\min_{x\in P} \lan h_2, x\ran-\min_{x\in P} \lan h_3, x\ran$;
\item[(b)] $l_1(AP)$ does not depend on the order of the rows in $A$;
\item[(c)] $l_1(AP)=l_1\left(BP\right),$ where $B=\begin{bmatrix}h_1\\h_2\\-(h_1+h_2+h_3)\end{bmatrix}$.
\end{itemize}
\end{Lemma}
\begin{proof}
Part (a) follows directly form the definition of $l_1(P)$:
\begin{align*}
l_1(AP)&=\max_{x\in P} \lan e_1+e_2+e_3, Ax\ran-\min_{x\in P} \lan e_1, Ax\ran-\min_{x\in P} \lan e_2, Ax\ran-\min_{x\in P}\lan e_3, Ax\ran\\
&=\max_{x\in P} \lan A^T(e_1+e_2+e_3), x\ran-\min_{x\in P} \lan A^Te_1, x\ran-\min_{x\in P} \lan A^Te_2, x\ran-\min_{x\in P}\lan A^{T}e_3, x\ran\\
&=\max_{x\in P} \lan h_1+h_2+h_3, x\ran-\min_{x\in P} \lan h_1, x\ran-\min_{x\in P} \lan h_2, x\ran-\min_{x\in P} \lan h_3, x\ran.
\end{align*}
Part (b) is straightforward. We check (c) using (a):
\begin{align*}
l_1\left(BP\right)&=\max_{x\in P} \lan-h_3, x\ran-\min_{x\in P} \lan h_1, x\ran-\min_{x\in P} \lan h_2, x\ran-\min_{x\in P}\lan -h_1-h_2-h_3, x\ran\\
&=\max_{x\in P}\lan h_1+h_2+h_3, x\ran -\min_{x\in P} \lan h_1, x\ran-\min_{x\in P} \lan h_2, x\ran-\min_{x\in P} \lan h_3, x\ran\\
&=l_1(AP).\\
\end{align*}
\end{proof}

\vspace{.1cm}

\begin{Definition}~\label{D:Minkowski-reduced} Let $P\subset\R^d$ be a convex body and let $\Vert h\Vert=\w_h(P)$. We say that a basis $(h_1,\dots,h_d)$ of $\Z^d$ is
{\it Minkowski reduced} if each $h_i$ is the shortest lattice vector such that $(h_1, . . . , h_i)$ can be extended to a basis of $\Z^d$ for $i=1,\dots,d$.
\end{Definition}
The norms $\mu_i=\mu_i(P):=\Vert h_i\Vert$ of the vectors in a Minkowski reduced basis $(h_1,\dots,h_d)$ of $\Z^d$ only depend on $P$ and are a version of  the successive minima of $K=(P+(-P))^\circ$, see~\cite{HarSopr} for details.

Clearly, a Minkowski reduced basis is reduced. The next example demonstrates that a reduced basis does not need to be Minkowski reduced. 
\begin{Example}\label{E:reduced-not-Mink-reduced}
Let $P\subset\R^3$ be the convex hull of  the set
$$\{(-1,0,1), (1,-1,0),  (-1,1,0), (1,1,-1), (-1,-1,1)\}
$$ and $(e_1,e_2,e_3)$ be the standard basis of $\Z^3$.
Consider the norm  $\Vert h\Vert=\w_h(P)$. Then $\Vert e_1\Vert=\Vert e_2\Vert =\Vert e_3\Vert=2$, $\Vert e_1+e_2\Vert=4$,  and $\Vert e_1-e_2\Vert=4$.
For $m,n\in Z$ we have
\begin{align*}
\Vert me_1+ne_2+e_3\Vert&=\max\{1-m,m-n,n-m,m+n-1,1-m-n\}\\
&-\min\{1-m,m-n,n-m,m+n-1,1-m-n\}\\
&\geq (m+n-1)-(1-m-n)=2(m+n)-2,
\end{align*}
and we conclude that $\Vert me_1+ne_2+e_3\Vert\geq 2$ if $m+n\geq 2$. Similarly,
$$\Vert me_1+ne_2+e_3\Vert\geq (1-m-n)-(m+n-1)=2-2(m+n)\geq 2 
$$
whenever $m+n\leq 0$. If $m+n=1$, we have  
$$\Vert me_1+ne_2+e_3\Vert\geq (m-n)-(n-m)=2(m-n)=2(2m-1)\geq 2
$$
if  $m\geq 1$. Similarly, $\Vert me_1+ne_2+e_3\Vert\geq (n-m)-(m-n)=2(n-m)=2(1-2m)\geq 2$ for $m\leq 0$.
We have checked that the standard basis is reduced with respect to $P$. However,  
$\Vert e_1+e_2+2e_3\Vert=1$ and $(e_1+e_2+2e_3, e_2, e_3)$ is a basis of $\Z^3$. We conclude that the standard basis is not Minkowski reduced. 
\end{Example}

Although a reduced basis is not necessarily Minkowski reduced, one can pass from a reduced basis to a Minkowski reduced one.
It was shown  in Theorems 2.3 and  3.3 of~\cite{HarSopr}  that if $(h_1,h_2,h_3)$ is a reduced basis of $\Z^3$  then $\Vert ah_1+bh_2\Vert\geq \Vert h_2\Vert$ for all $a,b\in\Z$ with $b\neq 0$ and 
$$\Vert ah_1+bh_2+ch_3\Vert\geq \Vert h_3\Vert {\rm\ for\ all \  } (a,b,c)\in\Z^3 {\rm\ with\ }c\neq 0, 
$$
except, possibly,  for two vectors $\pm u$ of the smallest norm $\Vert u\Vert=\Vert -u\Vert=\w_u(P)$ in   
$$E=\{ah_1+bh_2+ch_3: |a|=|b|=1, |c|=2\}.$$ 
Furthermore, two vectors of the smallest norm among $\{h_1,h_2,u\}$, 
written in the order of increasing norm,  together with $h_3$, form a Minkowski reduced basis and $\mu_3(P)=\Vert h_3\Vert$.

This provides us with an equivalent definition of a Minkowski reduced basis of $\Z^3$: A basis $(h_1,h_2,h_3)$ of $\Z^3$ is {\it Minkowski reduced} if $h_1\in\Z^3$ is a vector of smallest norm, $h_2\in \Z^3$ is the vector of smallest norm among vectors that are not multiplies of $h_1$, and $h_3\in\Z^3$ is a vector of smallest norm  such that $(h_1,h_2,h_3)$ is a basis of $\Z^3$.


 \begin{Lemma}\label{L:ls-geq-n} Let $(h_1,h_2,h_3)$ be a basis of $\Z^3$ which is Minkowski reduced with respect to $P$. Then $\ls(P)\geq \Vert h_3\Vert$.
  \end{Lemma}
\begin{proof}
Let $A\in{\rm GL}(3,\Z)$ be such that $\ls(P)=l_1(AP)$. The rows $r_1,r_2,r_3$ of $A$ form a basis of $\Z^3$ and hence  for one of them, say, $r_3$, we have $\Vert r_3\Vert\geq \Vert h_3\Vert$. Using parts (2) and (3) of Lemma~\ref{L:applyA} we conclude
$$\ls(P)=l_1(AP)\geq \w_{e_3}(AP)=\w_{r_3}(P)=\Vert r_3\Vert\geq\Vert h_3\Vert.
$$
\end{proof}

\begin{Lemma}\label{L:applyA-reduced} Suppose that $(h_1,h_2,h_3)$ is a basis of $\Z^3$ which is (Minkowski) reduced with respect to $P$ and let $A$ be a matrix with rows $h_1, h_2,$ and $h_3$.
Then the standard basis is (Minkowski) reduced with respect to $AP$.
  \end{Lemma}
\begin{proof} This follows directly from the definition of a (Minkowski) reduced basis and part (2) of Lemma~\ref{L:applyA} where we showed that 
$\w_{e_i}(AP)=\w_{h_i}(P)$ for $i=1,2,3$. 
\end{proof}

\section{Empty lattice triangles and parallelograms}
Recall that a lattice polytope $P$ is {\it empty} if its  only lattice points are its vertices.  It is an easy exercise to show that up to lattice equivalence the only empty lattice polygons $P\subset\R^2$ are the unit square $[0,1]^2$ and the standard 2-simplex. This, in particular, implies that a lattice parallelogram is empty if and only if its area equals 1. It was shown in~\cite{White} that up to lattice equivalence empty lattice tetrahedra in $\R^3$ are of the form
\begin{equation}
T_{pq}=\begin{bmatrix}1&0&0&p\\0&1&0&q\\0&0&1&1\end{bmatrix}, \label{e:empty-tetra}
\end{equation}
where $p$ and $q$ are non-negative integers satisfying $\gcd(p,q)=1$. Note that the columns of this matrix are the vertices of $T_{pq}$. (We will be using this notation throughout the paper: The polytope written as a matrix is the convex hull of its column vectors.) 

It was shown in~\cite{ACKS}  that if  $q\geq p^2-p$ and $q\geq 2$ then $\ls(T_{pq})=\lfloor \frac{q-2}{p+1}\rfloor+3$, and this computation was also extended~\cite{ACKS} to the case of dimension $n$.  Additional cases were computed in~\cite{Abdul-thesis}. 

Let $P\subset\R^3$ be an empty lattice polytope. It was shown in~\cite{Scarf} that any empty lattice polytope $P\subset\R^3$ has lattice width one and hence, using lattice equivalence, we can assume that $\w_{e_1}(P)=1$.  Denote $\Pi_0=P\cap\{x=0\}$ and $\Pi_1=P\cap\{x=1\}$. Since $P$ is empty,  each of $\Pi_0$ and $\Pi_1$ is  lattice equivalent to a point, a primitive segment, or an empty lattice polygon, which is either the standard simplex or the unit square $[0,1]^2$.  In other words, each of the two bases is either a lattice parallelogram of area 1 or is properly contained in such a parallelogram.

Note that the unimodular map
\begin{equation}\label{e:shift-layer}
A\begin{bmatrix}x\\y\\z\end{bmatrix}=\begin{bmatrix}1&0&0\\a&1&0\\b&0&1\end{bmatrix}\begin{bmatrix}x\\y\\z\end{bmatrix}
\end{equation}
fixes each point in the plane $x=0$ and shifts the plane $x=1$ by the vector $(0,a,b)$. Hence, composing such maps with lattice translations, one can use affine unimodular maps to independently shift $\Pi_0$ and $\Pi_1$ in the planes $x=0$ and $x=1$ correspondingly. Note that if, say,
$\Pi_0$ is a vertex, we can apply a unimodular map that maps $\Pi_1$ inside the unit square and then shift $\Pi_0$ so that $P\subset[0,1]^3$.

We next prove a few observations that  will allow us in Theorem~\ref{T:empty-polytopes} to treat empty lattice triangles and parallelograms as  lattice segments connecting two of  their vertices, which we will refer to as the main vertices.  

\begin{Proposition}\label{P:parallelograms} Let  $\Pi\subset\R^2$ be an empty parallelogram, an empty lattice triangle, or a primitive segment. Then one can shift $\Pi$ so that one of its  vertices is at the origin and the entire $\Pi$ is contained in the first or in the second quadrant.
\end{Proposition}

\begin{proof} The statement is clearly true when $\Pi$ is a lattice segment, so assume $\Pi$ is either an empty parallelogram or an empty lattice triangle.
Shift $\Pi$ so that its lowest vertex is at the origin. If there are two vertices at the lowest level, shift $\Pi$ so that one of them is at the origin. Let $(a,b)$ and $(c,d)$ be the two vertices adjacent to the vertex at the origin.
If $\Pi$ is not entirely in the first or in the second quadrant, then we may assume $a>0, b\geq 0$, while $c<0$ and $d\geq 0$. Since $\Pi$ is of area 1 we have $1=ad-bc\geq d+b\geq 0$, which implies that either $b=0, d=1$ or $b=1, d=0$. 
(Note that we cannot have $b=d=0$.)
In the first of these cases, we have  $a=1$ and we can shift $\Pi$ left by 1 so that one of its vertices is at the origin and $\Pi$ is contained in the second quadrant, see Figure~\ref{F:parallelograms}.
In the second case,  we have $c=-1$ and we shift $\Pi$ right by 1 unit. 
\end{proof}
\begin{figure}[h]
\begin{center}
\includegraphics[scale=.5]{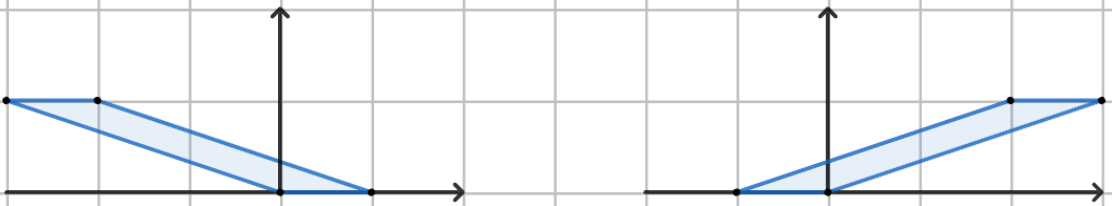} 
\caption{The two cases in Proposition~\ref{P:parallelograms}.}
\label{F:parallelograms} 
\end{center}
\end{figure}

\begin{Definition}
Let $\Pi$ be an empty lattice parallelogram, an empty lattice triangle, or a primitive segment.
If we can shift $\Pi$ so that one of its vertices is at the origin and $\Pi$ is contained in the first quadrant,  we say that $\Pi$ is {\it positively oriented}; if we can shift $\Pi$ so that one of its vertices is at the origin and  $\Pi$ is contained in the second quadrant, we say that $\Pi$ is {\it negatively oriented}. 
\end{Definition}

Note that it is possible for $\Pi$ to be both positively and negatively oriented, for example, this is the case if $\Pi$  is a translation of the unit square $[0,1]^2$. If $\Pi$ is a primitive segment, its orientation coincides with the sign of its slope, with horizontal and vertical segments being both positively and negatively oriented.

\begin{Definition}
Let $\Pi$ be either an empty lattice triangle or an empty lattice parallelogram. We say that $\Pi$ {\it possesses main vertices} if one can pick two vertices in $\Pi$ such that
the maximum and minimum of $x,y$, $x+y$, $x-y$ over $\Pi$ are attained on these two vertices, which we then call the {\it main vertices} of $\Pi$. 
In the case when $\Pi$ is a lattice segment, we say that $\Pi$ {\it possesses main vertices} and we call both of its vertices  the {\it main vertices}.
\end{Definition}

\begin{Proposition}\label{P:columns_ordered} 
Let $a,b,c,d$ be non-negative integers and let $\det\begin{bmatrix}a&b\\c&d\end{bmatrix}=\pm 1$. Then 
\begin{itemize}
\item[(i)] $a\geq b, c\geq d$ or
\item[(ii)] $a\leq b, c\leq d$ or
\item[(iii)] $\begin{bmatrix}a&b\\c&d\end{bmatrix}=\begin{bmatrix}1&0\\0&1\end{bmatrix}$.
\end{itemize}
\end{Proposition}

\begin{proof}
Suppose that $a> b$ but $c<d$. Then 
$\pm 1=ad-bc> bd-b(d-1)=b,$
so $b=0$ and  $a=d=1$. Also, $c\leq d-1=0$ implies $c=0$, so  $\begin{bmatrix}a&b\\c&d\end{bmatrix}=\begin{bmatrix}1&0\\0&1\end{bmatrix}$.
The case $a< b$, $c>d$ is handled similarly.
\end{proof}

The next proposition is illustrated in Figure~\ref{F:main-vertices}.

\begin{Proposition}\label{P:main-vertices-parallelogram} Let $\Pi\subset\R^2$ be an empty lattice parallelogram which is not contained in a lattice translation of the unit square.
Then $\Pi$ possesses main vertices.
\end{Proposition}

\begin{proof}
By reflecting $\Pi$ in the $y$-axis and translating it, if necessary, we can assume that $\Pi$ is positively oriented and has vertices $(0,0), (a,b), (c,d), (a+c,b+d)$, where $a,b,c,d\geq 0$ and $\Pi$ is not the unit square. Clearly, the minimum and the maximum  of $x+y$, $x$, and $y$ are attained at $(0,0)$ and $(a+c,b+d)$.  By Proposition~\ref{P:columns_ordered}, reflecting $\Pi$ in the line $x=y$ if necessary,  we can assume that $a\geq b, c\geq d$.  Hence 
$$0\leq a-b\leq (a-b)+(c-d)=(a+c)-(b+d)$$ 
and $0\leq c-d\leq (a+c)-(b+d)$, so the minimum of $x-y$ is attained at $(0,0)$, while its maximum is attained at $(a+c,b+d)$.
We have checked that $(0,0)$ and $(a+c,b+d)$ are the main vertices of $\Pi$.
\end{proof}

\begin{figure}
\begin{center}
\includegraphics[scale=.3]{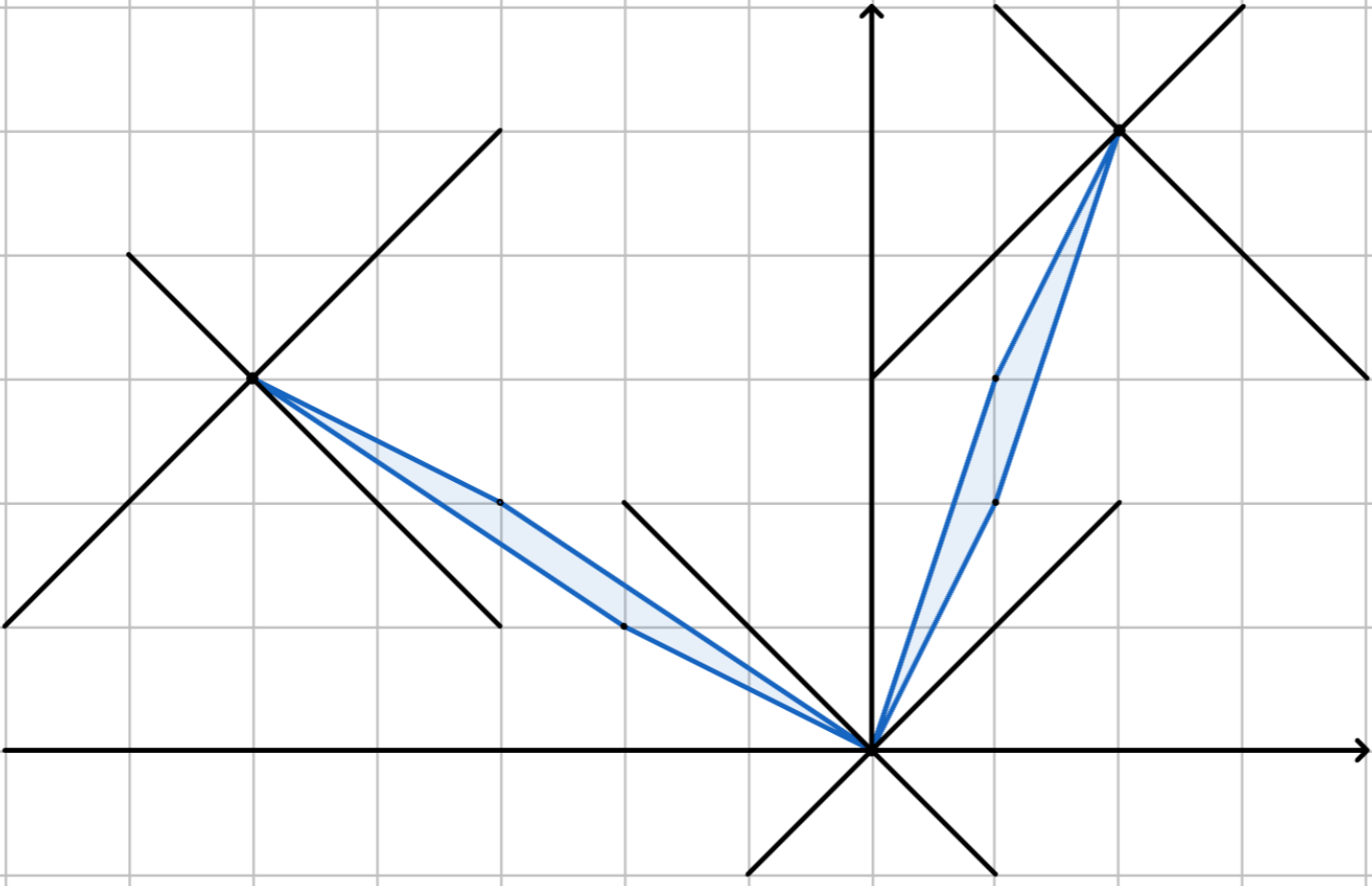}
\caption{Main vertices of parallelograms of area 1.}\label{F:main-vertices}
\end{center}
\end{figure}

\begin{Proposition}\label{P:main-vertices-triangle} Let $\Pi\subset\R^2$ be an empty lattice traingle which is not contained in a lattice translation of the unit square.
Then $\Pi$ possesses main vertices.
\end{Proposition}
\begin{proof} By reflecting $\Pi$ in the $y$-axis and translating it, if necessary, we can assume that $\Pi$ is positively oriented and has vertices $(0,0), (a,b), (e,f)$, where $a,b,e,f\geq 0$ and $\Pi$ is not contained in the unit square.  By applying Proposition~\ref{P:columns_ordered} to $\begin{bmatrix}a&e\\b&f\end{bmatrix}$ and, if necessary, swapping points $(a,b)$ and $(e,f)$, we can assume that
$a\leq e$ and $b\leq f$. Defining $c:=e-a\geq 0$ and $d:=f-b\geq 0$ we get $\Pi=\conv\{(0,0), (a,b), (a+c, b+d) \}$. The rest of the proof proceeds as in Proposition~\ref{P:main-vertices-parallelogram}, where the missing point $(c,d)$ does not make a difference. 
\end{proof}

\section{Lattice Size of Empty Lattice Polytopes}\label{S:main}

In this section we prove our main result that leads to a basis reduction algorithm for computing the lattice size $\ls(P)$ of an empty lattice polytope $P\subset\R^3$.

 \begin{Theorem}\label{T:empty-polytopes} Let $P\subset\R^3$ be an empty lattice polytope. Then there exists a  basis which is Minkowski reduced with respect to $P$ and computes the lattice size  $\ls(P)$.
 \end{Theorem}
 
 \begin{proof}
 By applying a unimodular change of variables we can assume that the standard basis $(e_1,e_2,e_3)$ is Minkowski reduced with respect to $P$. Since $P\subset\R^3$ is an empty lattice polytope, by Howe's theorem from~\cite{Scarf} we have $\w(P)=1$ and hence  $\Vert e_1\Vert=1$, $\Vert e_2\Vert=m$, and $\Vert e_3\Vert=n$, where $1\leq m\leq n$.
 
 If $n=1$, we also  have $m=1$, and hence $P$ is contained in the unit cube. We then have three options: $P$ is lattice equivalent to the standard simplex, in which case $\ls(P)=1$;
 $P$ is the unit cube, in which case $\ls(P)=3$; and anything in between, in which case we have $\ls(P)=2$.

In what follows we assume that $n>1$.  Denote  $\Pi_0=P\cap\{x=0\}$ and $\Pi_1=P\cap\{x=1\}$. As discussed in the beginning of Section 3, each of the two bases is either a lattice
parallelogram of area 1 or is properly contained in such a parallelogram. Note that if one of  the $\Pi_i$ is a vertex, we can unimodularly map the other inside the unit square and we would have $n=1$.  Hence we can assume that each of the $\Pi_i$ is a primitive segment, a lattice triangle of  area 1/2, or a lattice parallelogram of area 1.

 \vspace{.1cm}
 {\bf\noindent Case 1.} Assume first that $\Pi_0$ and $\Pi_1$ are of opposite orientation, say, $\Pi_0$ is positively oriented while $\Pi_1$ is negatively oriented. Additionally, assume 
that neither of the  $\Pi_i$ is a translation of the unit square or of its half, so by Propositions~\ref{P:main-vertices-parallelogram} and~\ref{P:main-vertices-triangle} both $\Pi_0$  and $\Pi_1$ possess main vertices.
 
By translating $P$ if necessary, we can assume that  the main vertices  in $\Pi_0$ are $(0,0,0)$ and $(0,\alpha,\beta)$ for some $\alpha\in[0,m]$ and $\beta\in[0,n]$. Then, if necessary, apply to $P$ a matrix of the form (\ref{e:shift-layer}) to translate $\Pi_1$ in the plane $x=1$ so that its vertices are of the form $(1,\gamma,0)$ and $(1,0,\delta)$ for some $\gamma\in[0,m]$ and $\delta\in[0,n]$. 
This may only bring $\w_{e_2}(P)$ and $\w_{e_3}(P)$ down, but since the standard basis  is Minkowski reduced both widths will stay the same.

Since $\Vert e_2\Vert=m$ and $\Vert e_3\Vert=n$, and the maximum and minimum of $y$ and $z$ are attained on the main vertices, we have $\alpha=m$ or $\gamma=m$, and $\beta=n$ or $\delta=n$.

Note that we can swap the two layers $x=0$ and $x=1$, swapping $\alpha$ with $\gamma$ and $\beta$ with $\delta$ by applying
$$
\begin{bmatrix}-1&0&0\\ 0&1&0\\ \delta-\beta&0&-1\end{bmatrix}\begin{bmatrix} 0&0&1&1\\ 0&\alpha&\gamma&0\\ 0&\beta&0&\delta\end{bmatrix}+\begin{bmatrix}1\\ 0\\ \beta\end{bmatrix}=
\begin{bmatrix} 1&1&0&0\\ 0&\alpha&\gamma&0\\ \beta&0&\delta&0\end{bmatrix}
$$  
Hence we can assume that we have $\delta=n$, that is, the main vertices of $P$ are the columns of
$$
\begin{bmatrix} 0&0&1&1\\ 0&\alpha&\gamma&0\\ 0&\beta&0&n \end{bmatrix},
$$
which implies that $\gamma>0$ for otherwise  $[(1,\gamma,0),(1,0,\delta)]$ would not be primitive as $\delta=n>1$.

 Since the standard basis is reduced we have $\Vert-\gamma e_1+e_2+e_3\Vert\geq n$, which gives
 $$\max\{0,\alpha+\beta,n-\gamma\}-\min\{0,\alpha+\beta,n-\gamma\}\geq n.
 $$   
(Note that since we are minimizing and maximizing $(y+z)$ on each of the layers  here, only the main vertices matter.) Since $\gamma>0$, this condition implies $\alpha+\beta\geq n$.  Hence $\alpha>0$ for otherwise $\beta=n$ and $[(0,0,0),(0,\alpha,\beta)]$ is not primitive.
 
 \vspace{.1cm}
 {\bf\noindent Case 1a.} Suppose first that $\alpha+\beta=n$. Then $l_1(P)=n+1$. Let's show that in this case we have $\ls(P)=n+1$. For this, according to part (5) of Lemma~\ref{L:applyA}
 we need to look for integer directions $(a,b,c)$ with $c\geq 0$ such that the lattice width of $P$ in such directions is at most $n$. 
We have
$$\max\{0,b\alpha+c\beta, a+b\gamma,a+cn\}-\min \{0,b\alpha+c\beta, a+b\gamma,a+cn\}\leq \w_{(a,b,c)}(P)\leq n.
$$    
Hence $cn-b\gamma\leq n$. Also, plugging in $\beta=n-\alpha$ we get $b\alpha+c(n-\alpha)\leq n$.
   If $c>1$ we get $n-b\gamma<cn-b\gamma\leq n$, so $b\gamma>0$ and hence $b>0$ since we have $\gamma>0$. Also, if $\alpha=n$ then $\beta=0$ and hence $[(0,0,0),(0,\alpha,\beta)]$ is not primitive.
It follows that $\alpha<n$, so we get
 $$b\alpha+(n-\alpha)<b\alpha+c(n-\alpha)\leq n,
 $$
which  implies $b\alpha<\alpha$. Using $\alpha>0$  we conclude that $b<1$,  which contradicts $b>0$ obtained above. We have shown that $c=0$ or 1. This implies that if $T(P)\subset n\Delta $ then the last coordinate of each of the rows of the corresponding unimodular matrix $A$ has 
 last coordinate equal to 0, 1, or $-1$. By Lemma~\ref{L:applyA}, the same applies to the sum of any two rows of $A$, as well as to the sum of all three rows. Hence up to permuting the rows, we have two options for the last column of $A$:
 $$\pm\begin{bmatrix} 0\\ 0\\1\end{bmatrix}\  {\rm and}\ \ \pm\begin{bmatrix} 0\\ -1\\1\end{bmatrix},
 $$  
where by Lemma~\ref{L:sum-of-rows} we only need to consider the first option. Denote the rows of $A$ by $h_1, h_2,$ and $h_3$. Since the standard basis is reduced and $h_3$ is of the form
$ae_1+be_2\pm e_3$ we have 
$$\w_{h_3}(P)=\Vert h_3\Vert \geq n.$$ 
Together with $T(P)\subset n\Delta$, this implies that $(0,0,n)\in  T(P)$. We also have $\Vert h_1+h_2+h_3\Vert \geq n$ and hence $(0,0,0)\in T(P)$.
Hence the entire segment $[(0,0,0), (0,0,n)]$ is contained in $T(P)$. Since $P$ is empty and $n>1$ we conclude that there is no such map $T$ and hence $\ls(P)=n+1$.

 \vspace{.1cm}
 {\bf\noindent Case 1b.} We can now assume that $\alpha+\beta\geq n+1$, which implies $\beta>0$. From above, we also have $\delta=n, \alpha>0, \gamma>0$. We then have $l_1(P)=\alpha+\beta$.  Also, $\gamma+\delta=n+\gamma\geq n+1$.  If we swap the layers $x=0$ and $x=1$, as explained above, we can unimodularly inscribe $P$ inside  $(\gamma+\delta)\Delta$. We now show that $\ls(P)=\min\{\alpha+\beta,\gamma+\delta\}$. Let $(a,b,c)$ be an integer direction with $c\geq 0$ such that  $\Vert (a,b,c)\Vert<\min\{\alpha+\beta,\gamma+\delta\}$. 
Then
$$\max\{0,b\alpha+c\beta, a+b\gamma,a+cn\}-\min \{0,b\alpha+c\beta, a+b\gamma,a+cn\}<\min\{\alpha+\beta,\gamma+\delta\}.
$$  
If $c>1$ we get $b\alpha +\beta<b\alpha+c\beta<\alpha+\beta$, so $b\alpha<\alpha$, and since $\alpha>0$ we conclude $b\leq 0$. Using this, we get
$2n\leq cn\leq cn-b\gamma<n+\gamma$, which implies $n<\gamma$ and this contradiction  proves $c\leq 1$.
If $c=1$, we have $b\alpha+\beta<\alpha+\beta$ and hence $b<1$. Also, $n-b\gamma<n+\gamma$, so $b>-1$, and we conclude that $b=0$. We have checked that $c=0$ or 1 and that in the case when $c=1$ we also have $b=0$.

As before, if we have $l_1(T(P))<\min\{\alpha+\beta,\gamma+\delta\}$ we can assume that the last column of the corresponding unimodular matrix $A$ is the transpose of $\begin{bmatrix}0&0&\pm 1\end{bmatrix}$.
Let the rows of $A$ be $h_1, h_2$, and $h_3$. Then the last coordinate of $h_3, h_1+h_3$, and $h_2+h_3$ is 1 so by what we just showed and by Lemma~\ref{L:applyA} the second component of each of these vectors is 0, so the entire second column of a unimodular matrix $A$ would consist of all zeroes. This contradiction proves that in this case $\ls(P)=\min\{\alpha+\beta,\gamma+\delta\}$.

 \vspace{.1cm}
 {\bf\noindent Case 2.} We next consider the situation where at least one of the $\Pi_i$ does not possess main vertices, that is, $\Pi_i$, up to a lattice shift, is either  the unit square or a triangle which is  half of the unit square.
If this happens for both $\Pi_0$ and $\Pi_1$, we have $n=1$.
Let us assume $\Pi_1$ is either the unit square $$\conv\{(1,0,0), (1,0,1), (1,1,0), (1,1,1)\}$$ or a triangle obtained by dropping one of its vertices, and that the main vertices in $\Pi_0$ are $(0,\alpha,0)$ and $(0,0,\beta)$.
Then since $n>1$ we have $\beta=n$. If $m\geq 1$ we also have $\alpha=m$.  If $m=1$, we cannot have $\alpha=0$ since then $[(0,\alpha,0), (0,0,\beta)]$ would not be primitive, so we again have $\alpha=m$.

Assume first that $\Pi_1$ is the unit square. We then have 
$$\max\{n,m,n-1, n+1\}-\min\{n,m,n-1,n+1\}\geq \Vert (n-1)e_1+e_2+e_3\Vert\geq  n.$$ 
If $n>m$ we get $n+1-m\geq n$, which implies $m=1$. If $n\geq 3$ we get $l_1(P)=n$ and hence by Lemma~\ref{L:ls-geq-n} we conclude that $\ls(P)=n$. If $n=2$,  we have $l_1(P)=3$. 
Let  $(a,b,c)$ be an integer direction  such that the lattice width of $P$ in this direction is at most 2. As when showing above that $\ls(P)=n+1$ under the assumption that $\alpha+\beta=n$, here one can easily check that $|c|\leq 1$ and there is no unimodular map $T$ with $l_1(T(P))=n$. If $\Pi_1$ is a triangle which is  half of the unit square, the computation is similar.

 \vspace{.1cm}
 {\bf\noindent Case 3.} It remains to consider the case when $\Pi_0$ and $\Pi_1$ are of the same orientation. We can then assume that $\Pi_0$ and $\Pi_1$ are both oriented positively. Shift $\Pi_0$ so that its main vertices are
  $(0,0,0)$ and $(0,\alpha,\beta)$, where $0\leq \alpha\leq m$ and $0\leq  \beta\leq n$.  Shift $\Pi_1$ so that its main vertices are at
 $(1,\gamma,\delta)$ and $(1,m,n)$, where $0\leq \gamma\leq m$ and $0\leq  \delta\leq n$.
   
We have $\alpha=m$ or $\gamma=0$ since otherwise we can shift $\Pi_1$ in $x=1$ and decrease $\Vert e_2\Vert$.
Similarly,  we have $\beta=n$ or  $\delta=0$. Switching the layers $x=0$ and $x=1$, if necessary, we can assume that $\beta=n$.

Since the standard basis is reduced, we have $\Vert e_3-e_2\Vert \geq n$, and this gives
$$\max\{0, n-\alpha,\delta-\gamma\}-\min\{0, n-\alpha,\delta-\gamma\}\geq n.$$
Since $n-\alpha\geq 0$ this implies 
\begin{equation}
\max\{n-\alpha,\delta-\gamma\}-\min\{0,\delta-\gamma\}\geq n.\label{e:eqn}
\end{equation}   
If this translates to $n-\alpha\geq n$, then $\alpha=0$, which implies that $n=\beta=1$. 

If (\ref{e:eqn})gives $\delta-\gamma\geq n$ then $\delta=n$ and $\gamma=0$, so $m=1$ and $\Pi_1=[(1,0,n), (1,1,n)]$ is negatively oriented, which we have covered before.

Finally, if (\ref{e:eqn}) is equivalent to $n-\alpha-\delta+\gamma\geq n$ we get $\alpha+\delta\leq \gamma$. Since $\alpha=m$ or $\gamma=0$, we have $\gamma\leq \alpha$ and hence
$\gamma+\delta\leq \alpha+\delta\leq \gamma$, so $\delta=0$ and $\alpha=\gamma$. If $\alpha=\gamma=m$, we get $[(1,m,0), (1,m,n)]\subset\Pi_1$ and hence $n=1$.
If $\alpha=\gamma=0$ then $[(0,0,0), (0,0,n)]\subset\Pi_0$, so $n=1$.

 \end{proof}

An algorithm for finding a basis which is reduced with respect to given $P\subset\R^3$ is provided in Algorithm 3.14 of~\cite{HarSopr} and its complexity is analyzed in Theorem 3.17 in the same paper. This algorithm is a modified version of Algorithm 3.11 of~\cite{HarSopr} where some termination conditions were added to make the complexity analysis possible.

Based on the proof of Theorem~\ref{T:empty-polytopes}, we can compute the lattice size of an empty lattice polytope $P\subset\R^3$ using the following algorithm.

 \begin{Algorithm}\label{A:compute-ls} Let $P\subset\R^3$ be an empty lattice polytope. Use Algorithm 3.11 or 3.14 from~\cite{HarSopr} to find a basis of $\Z^3$ which is reduced with respect to $P$. (That is, to $K=(P+(-P))^\circ$ in the more formal language used in~\cite{HarSopr}).  Next, find the corresponding Minkowski reduced basis $(h_1,h_2,h_3)$, as explained in the paragraph following Example~\ref{E:reduced-not-Mink-reduced}
 and pass to $Q=AP$ where $A$ is a unimodular matrix with rows $h_1$, $h_2$, and $h_3$. Consequently, by Lemma~\ref{L:applyA-reduced}, the standard basis is Minkowski reduced with respect to $Q$.  For each of the eight matrices 
 $$
 \Sigma=\begin{bmatrix}\pm1&0&0\\  0&\pm1&0\\ 0&0&\pm 1\end{bmatrix}
 $$
 shift the bases $\Pi_0$ and $\Pi_1$ of $\Sigma Q$ in the layers $x=0$ and $x=1$ so that each of them is in the corresponding positive $(y,z)$-octant and touches both coordinate axes, and call the obtained polytope $R$. Then $\ls(P)$ is the minimum of $l_1(R)$ over the eight matrices $\Sigma$.
 \end{Algorithm}

 \section{Another class of 3D polytopes of lattice width one}\label{S:another-class}

 In this section we describe another class of lattice polytopes $P\subset\R^3$ of lattice width one  for which there exists a reduced basis that computes its lattice size $\ls(P)$.  
 
In Proposition~\ref{P:det-not-one}, we are working with the plane case, so now $\Delta\subset\R^2$ is the standard 2-simplex. We show that if one replaces unimodular matrices in the definition of 
$\ls(P)$ with nonsingular integer matrices, then this will define the same object.

 \begin{Proposition}\label{P:det-not-one} Let $P\subset\R^2$ be a lattice polygon. Then for any integer nonsingular matrix $A$ of size 2 we have $l_1(AP)\geq \ls(P)$.
 \begin{proof}
Let $A=UDV$ be the Smith normal form for $A$. That is, $U$ and $V$ are unimodular matrices and $D$ is an integer diagonal matrix with positive entries on the diagonal.
Then, since $U$ and $V$ are unimodular, we have $\ls(UDVP)=\ls(DVP)$ and $\ls(VP)=\ls(P)$. Since $D$ is diagonal with positive integer entries on the diagonal, assuming that $VP$ contains the origin, 
we have $VP\subset DVP$ and hence $\ls(VP)\leq\ls(DVP)$.  Hence
$$l_1(AP)\geq \ls(AP) =\ls(UDVP)=\ls(DVP)\geq \ls(VP)=\ls(P).
$$
 \end{proof}
 \end{Proposition}
 
 \begin{Proposition}\label{P:lstwo}
 Let $P\subset\R^2$ be a polygon in the $(x,y)$-plane of $\R^3$. Let  $\lstwo(P)=l$, where $\Delta_2$ is the standard simplex in the $(x,y)$-plane.
 Then $\ls(P)=l$, where  $P$ is now considered as a subset of $\R^3$, and $\Delta$ is the standard simplex in $\R^3$. 
 \end{Proposition}
 
 \begin{proof} Suppose that for 
 $$A=\begin{bmatrix}a_{11}&a_{12}&a_{13}\\ a_{21}&a_{22}&a_{23}\\ a_{31}&a_{32}&a_{33}\end{bmatrix}\in{\rm{GL}(3,\Z)}$$ we have $l_1(AP)=l'<l$ and hence $AP+h\subset\l'\Delta$ for some $h=[b_1,b_2,b_3]^T\in\Z^3$. 

 Since $A$ is unimodular, at least one of the matrices $\begin{bmatrix}a_{11}&a_{12}\\ a_{21}&a_{22}\end{bmatrix}$, $\begin{bmatrix}a_{11}&a_{12}\\ a_{31}&a_{32}\end{bmatrix}$, and $\begin{bmatrix}a_{21}&a_{22}\\ a_{31}&a_{32}\end{bmatrix}$ is nonsingular. Hence we can assume that the determinant of $B:=\begin{bmatrix}a_{11}&a_{12}\\ a_{21}&a_{22}\end{bmatrix}$ is nonzero. Let $\pi:\R^3\to\R^2$ be the projection defined by $(x,y,z)\mapsto(x,y)$. 
 We have $\pi(AP+h)\subset\pi(l'\Delta)$ and hence $BP+[b_1,b_2]^T\subset l'\Delta_2$, where $\Delta_2$ is the standard simplex in the $(x,y)$-plane. Since $\lstwo(P)=l$ and $l'<l$, this contradicts the result of Proposition~\ref{P:det-not-one}. 
 \end{proof}
 
 \begin{Theorem}\label{T:another-class} Let $P\subset\R^3$ be a convex lattice polytope of lattice width one enclosed between the planes $x=0$ and $x=1$.  Denote the bases of $P$ by
 $$P_0=P\cap \{x=0\}\ \ {\rm and}\ \  P_1=P\cap\{x=1\}.$$ Consider $P_0$ and $P_1$ as subsets of the $(y,z)$-plane by ignoring the $x$-coordinate  and suppose that up to a lattice translation $P_1$ is contained in the convex hull of the interior lattice points of $P_0$.  Then $\ls(P)=\lstwo(P_0)$ and there exists a reduced basis that computes $\ls(P)$.
 \end{Theorem}


 \begin{proof}
 Use a unimodular transformation to shift $P_1$ in the plane $x=1$ so that its projection onto the plane $x=0$ along the $x$-axis is contained in the convex hull of the  interior lattice points of $P_0$. Let $A=\begin{bmatrix}a_{11}&a_{12}\\a_{21}&a_{22}\end{bmatrix}\in{\rm GL(2,\Z)}$ be a matrix whose rows form a basis of $\Z^2$ that is reduced with respect to $P_0$, that is, the standard basis of $\Z^2$ is reduced with respect to $AP_0$. Applying to $P$ the map defined by the matrix
$$\begin{bmatrix}1&0&0\\0&a_{11}&a_{12}\\0&a_{21}&a_{22}\end{bmatrix},
$$
 we can assume that $(e_2,e_3)$ is reduced with respect to $P_0$. Note that the projection of $P_1$ onto $x=0$ is still contained in the convex hull of the  interior lattice points of $P_0$.

 Since $\Vert e_1\Vert=1$, we have $\Vert e_1\Vert\leq \Vert e_2\Vert$ and $\Vert e_1\Vert\leq \Vert e_3\Vert$. Since $P_0$ contains the projection of $P_1$ 
 we have $\w_{e_2}(P)=\w_{e_2}(P_0)$ and $\w_{e_3}(P)=\w_{e_3}(P_0)$. Since  $(e_2,e_3)$ is reduced with respect to $P_0$, we have $\w_{e_3}(P_0)\geq \w_{e_2}(P_0)$ and 
 $\w_{ne_2+e_3}(P_0)\geq \w_{e_3}(P_0)$ for $n\in\Z$. Since $P_0$ is contained in the plane $x=0$ we have $\w_{me_1+ne_2+e_3}(P_0)=\w_{ne_2+e_3}(P_0)$ and 
$\w_{e_1\pm e_2}(P_0)=\w_{e_2}(P_0)$.
 
Hence
$$\Vert e_3\Vert=\w_{e_3}(P)=\w_{e_3}(P_0)\geq \w_{e_2}(P_0)=\w_{e_2}(P) =\Vert e_2\Vert.
$$   
Next, since $P_0\subset P$ we have
$$\Vert e_1\pm e_2\Vert=\w_{e_1\pm e_2}(P)\geq \w_{e_1\pm e_2}(P_0)=\w_{e_2}(P_0)=\w_{e_2}(P)=\Vert e_2\Vert.
$$
Finally, for $m,n\in\Z$,  we have
\begin{align*}
\Vert me_1+ne_2+e_3\Vert&=\w_{me_1+ne_2+e_3}(P)\geq \w_{me_1+ne_2+e_3}(P_0)\\
&=\w_{ne_2+e_3}(P_0)\geq \w_{e_3}(P_0)=\w_{e_3}(P)=\Vert e_3\Vert.
\end{align*}
We have checked that the standard basis is reduced with respect to the obtained $P$.

Denote $l=\lstwo(P_0)$.   By Proposition~\ref{P:lstwo} we have 
\begin{equation}\label{e:lsleql}
\ls(P)\geq \ls(P_0)=\lstwo(P_0)=l.
\end{equation}
Since basis $(e_2,e_3)$ is reduced with respect to $P_0$, we have $\lstwo(P_0)={\rm nls}_{\Delta_2}(P_0)$. (This is the main result of~\cite{HarSoprTier} that we explained in Section 2.)  Hence shifting and reflecting $P$ in the $(e_1,e_2)$ and $(e_1,e_3)$ planes, if necessary,  we can assume that $P_0\subset l\Delta$. 
 Since the projection of $P_1$ is contained in the convex hull of the interior lattice points of $P_0$, and $P_1$ is at the level $x=1$, we also have  $P_0\subset l\Delta$. It follows that $P\subset l\Delta$ and hence $\ls(P)\leq l$. Together with (\ref{e:lsleql}) this implies $\ls(P)=l$.  
 \end{proof}
 
We have shown that to compute the lattice size of a width one lattice polytope $P$ described in Theorem~\ref{T:another-class}  we need to compute the lattice size of $P_0\subset\R^2$. Let $(h_2,h_3)$ be a basis of $\Z^2$ reduced with respect to $P_0$  and  let $A$ be a matrix with rows $h_2$ and $h_3$. Then combining Theorem~\ref{T:another-class} with Theorem 2.7 from~\cite{HarSoprTier} we get $\ls(P)=\lstwo(P_0)={\rm nls}_{\Delta_2}(AP_0).$

\section{Counterexample}\label{S:experimentation}
In light of our results proved in Theorems~\ref{T:empty-polytopes} and~\ref{T:another-class} it is natural to ask whether it is true that for any lattice polytope $P\subset\R^3$ of lattice width one there exists a reduced basis that computes its lattice size $\ls(P)$. We next show that the answer to this question is negative.

%

\begin{Example}\label{E:counterexample} 

Let $P=\conv\{(0,2,5), (0,-2,5), (0,1,-6), (1,-8,5), (1,2,-5), (1,-4,-3)\}$.
Our first goal is to show that the standard basis $(e_1,e_2,e_3)$ is Minkowski reduced with respect to $P$. We have $\Vert e_1\Vert=1$, $\Vert e_2\Vert=10$ and $\Vert e_3\Vert=11$.
Hence it is enough to consider all $\Z^3$ bases $(h_1,h_2,h_3)$ with $\Vert h_i\Vert\leq 11$ and check that $\Vert h_1\Vert+\Vert h_2\Vert+\Vert h_3\Vert\geq 22$.
Shift $P$ so that its barycenter is at  the origin. Consider a ball $B\subset P$ centered at the origin and let its radius be $R$. Denote by $\Vert h\Vert_2$ the Euclidean norm of $h$.
Then if $\Vert h\Vert_2\geq \frac{11}{R}$ we have
$$\Vert h\Vert=\w_h(P)\geq \w_h(B)=R\cdot\Vert h\Vert_2\geq 11.
$$
This means that we only need to work with $h_i$ that satisfy $\Vert h_i\Vert\leq  \frac{11}{R}$. In  {\it Counterexample.mgm} in \cite{GitHub} we create a list of all primitive vectors $h\in \Z^3$ with the absolute value of components bounded from above by $\lfloor \frac{11}{R}\rfloor$ that also satisfy $\Vert h\Vert\leq 11$. We then use these vectors to build all possible matrices $A$ with rows $h_1,h_2,h_3$ from this list. If matrix $A$ is unimodular we check whether $\Vert h_1\Vert+\Vert h_2\Vert+\Vert h_3\Vert\geq 22$. This turns out to always be the case which allows us to conclude that the standard basis is indeed Minkowski reduced with respect to $P$.

For any basis $(h_1,h_2, h_3)$ which is reduced with respect to $P$ we then have $\Vert h_3\Vert=\mu_3(P)=11$, while
$\Vert h_1\Vert\leq 11$, and  $\Vert h_2\Vert\leq 11$. (See the discussion after  Definition~\ref{D:Minkowski-reduced}.)
We next show that the minimum of $l_1(AP)$ over all matrices $A$ whose rows form a reduced basis of $\Z^3$ is bounded below by 14.
For this we use the same list as above to build unimodular matrices $A$ with rows $h_i$ from the list and  check that we always have $l_1(AP)\geq 14$.


However,  if we apply to $P$ a unimodular map defined by 
$$A=\begin{bmatrix}
1&0&0\\2&-1&0\\7&2&1\end{bmatrix}$$
we get $l_1(AP)=13$. (Note that rows $h_1$, $h_2$, $h_3$ of this matrix satisfy $\Vert h_1\Vert=1, \Vert h_2\Vert=12, \Vert h_3\Vert=13$ and therefore basis $(h_1,h_2,h_3)$ is not reduced.) We have checked  that for this $P$ there is no reduced basis that computes its lattice size.
\end{Example}

We conclude that for $P\subset\R^3$ with $\w(P)=1$ there does not need to exist a reduced basis that computes its lattice size $\ls(P)$. 


%

\end{document}